\documentclass[12pt]{article}
\usepackage{amsmath, amsthm, amssymb}
\usepackage{url}

\newcommand{\NT}{\mathcal{NT}}
\newcommand{\Tr}{\mathcal{T}}
\newcommand{\M}{\mathcal{M}}
\newcommand{\F}{\mathbb{F}}

\newtheorem{thm}{Theorem}
\newtheorem{defn}[thm]{Definition}
\newtheorem{lem}[thm]{Lemma}
\newtheorem{cor}[thm]{Corollary}

\newtheorem{prop}{Proposition}

\title{On Parameters of Subfield Subcodes of Extended Norm-Trace Codes}
\author{Heeralal Janwa and Fernando L. Pi\~nero}

\begin{document}

\maketitle

\centerline{\scshape Heeralal Janwa }
\medskip
{\footnotesize
 \centerline{Department of Mathematics}
   \centerline{Faculty of Natural Sciences}
   \centerline{University of Puerto Rico -- R\'io Piedras Campus}
   \centerline{ San Juan, Puerto Rico, 00925 USA}
} 

\medskip

\centerline{\scshape Fernando L. Pi\~nero}
\medskip
{\footnotesize
 \centerline{Institute of Statistics and Computerized Information Systems}
   \centerline{Faculty of Business Administration}
   \centerline{University of Puerto Rico -- R\'io Piedras Campus}
   \centerline{ San Juan, Puerto Rico, 00925 USA}
}

\begin{abstract}

In this article we describe how to find the parameters of subfield subcodes of extended Norm--Trace codes. With a Gr\"obner basis of the ideal of the $\F_{q^r}$ rational points of the Norm--Trace curve one can determine the dimension of the subfield subcodes or the dimension of the trace code. We also find a BCH--like bound from the minimum distance of the original supercode.

\end{abstract}

\section{Background and Motivation}

Our aim is to understand subfield subcodes of extended Norm--Trace codes. These codes include AG codes from some quasi-Hermitian curves and Norm-Trace curves. BCH codes and binary classical Goppa codes are subfield subcodes of {RS} codes and  subfield subcodes of special AG codes of genus 0 curves (i.e. very special MDS codes). These subfield subcodes inherit good parameters, good automorphism groups and efficient encoding and decoding algorithms from their AG supercode. We have previously used Gr\"obner bases in \cite{JP14} to study subfield subcodes of Hermitian curves. We also found codes which are optimal or best known. Now we generalize these results to subfield subcodes of extended Norm-Trace codes. In this article we provide a Gr\"obener basis framework, and prove results that help us give explicit algorithms for computing the parameters of subfield subcodes of Extended Norm--Trace codes. These results are easily adaptable to encoding and decoding. We have implemented these algorithms  in symbolic software. We finish 
with some subfield subcodes of extended Norm--Trace codes which either have optimal parameters or are as good as any known code. 

We fix $q=p^l$ a prime power and $r > 1$ a positive integer. In addition $t$ is a prime power such that $\F_t \subseteq \F_{q^r}$. We denote the trace function of $\F_{q^r}$ over $\F_{t}$ by $\Tr_{\F_{q^r}/ \F_{t}  }$.  Note that $ \Tr_{\F_{q^r}/ \F_{t}  }(y) = \sum\limits_{i=0}^{r-1} y^{q^i}$. 

\section{Subfield Subcodes} For the material in this section, we refer the reader to \cite{S90} .

\begin{defn} 
Let $C$ be a code over $\F_{q^r}$ of length $n$. The subfield subcode of $C$ is defined as $$C|\F_t := C \cap \F_t^n.$$
The trace code of $C$ is defined as
$$ \Tr_{\F_{q^r}/ \F_{t}  }(C) := \{ (\Tr_{\F_{q^r}/ \F_{t}  }(c_1),\Tr_{\F_{q^r}/ \F_{t}  }(c_2),\ldots,\Tr_{\F_{q^r}/ \F_{t}  }(c_n)) | c \in C \}.$$
\end{defn}

Both $C|\F_t$ and $\Tr_{\F_{q^r}/ \F_{t}  }(C)$ are linear codes over $\F_t$ of length $n$.  In fact: 
\begin{prop}\cite{S90}[Delsarte's Theorem]
Let $C$ be  a code over $\F_{q^r}$. Then
$$(C|\F_t)^{\perp} = \Tr_{\F_{q^r}/ \F_{t}  }(C^{\perp}).$$
\end{prop}

The map $x \mapsto x^t$ is an automorphism of $\F_{q^r}$ which fixes $\F_t$ pointwise. One extends this map to a linear space as follows:

\begin{defn}\cite{S90}
Let $C$ be a code over $\F_{q^r}$ of length $n$.  Define $$ C^{(t)} := \{ (c_1^t,c_2^t,\ldots,c_n^t) | c  \in C \}.$$
\end{defn}

Stichtenoth \cite{S90} showed that $C|\F_t$ and $\Tr_{\F_{q^r}/ \F_{t}  }(C)$ may be seen as codes over $\F_{q^r}$.
\begin{prop}\cite{S90}
Suppose $q^r = t^{m}$. Let $C$ be a code over $\F_{q^r}$. Let $C^o := \bigcap\limits_{i=0}^{m-1} C^{(t^i)}$ and $C ^\wedge:= \sum\limits_{i=0}^{m-1} C^{(t^i)}$. Then  $C^o$ is the $\F_{q^r}$--linear code spanned by $C|\F_t$ over $\F_{q^r}$ and  $C^\wedge$ the $\F_{q^r}$-linear code spanned by $\Tr_{\F_{q^r}/ \F_{t}  }(C)$ over $\F_{q^r}$. Moreover $C^o$ and $C|\F_t$ have the same dimension and minimum distance. Likewise, $C^\wedge$ and $\Tr_{\F_{q^r} / \F_{t}}(C)$ also have the same dimension and minimum distance.
\end{prop}










\section{Extended Norm--Trace codes}
The extended Norm--Trace curve and their associated linear codes were introduced in \cite{BO07}. Extended Norm--Trace curves are defined in greater generality but for simplicity we fix $u | {\frac{q^r-1}{q-1}}$.

\begin{defn}

Suppose $V = \{ P_1, P_2, \ldots, P_n \} \subseteq \F_{q^r}^m$ is a finite set and $L$ is an $\F_{q^r}$ linear subspace of polynomials. The \emph{affine variety code} is defined as $$C(V, L) := \{  (f(P_1), f(P_2), \ldots, f(P_n)) \ | f \in L  \}.$$

\end{defn}

The Trace code of an affine variety code is expressed in a simple way.

\begin{lem}
Suppose $V = \{ P_1, P_2, \ldots, P_n \} \subseteq \F_{q^r}^m$ is a finite set and $L$ is an $\F_{q^r}$ linear subspace of polynomials. Then 
$$\Tr_{\F_{q^r}/ \F_{t}  }(C(V, L) )  = C(V, \Tr_{\F_{q^r}/ \F_{t}  }(L)) = C(V, \sum\limits_{i=0}^{m-1}(L^{(t^i)})) .$$ 
\label{lem:traceaffinecode}
\end{lem}

\begin{proof}
Let $c \in \Tr_{\F_{q^r}/ \F_{t}  }(C(V, L) )$. The vector $c$ is equal to a vector of the form $(\Tr_{\F_{q^r}/ \F_{t}}(f(P_1)), \Tr_{\F_{q^r}/ \F_{t}}(f(P_2)), \ldots, \Tr_{\F_{q^r}/ \F_{t}}(f(P_n)))$ for some $f \in L$.  Therefore the codeword $c = (g(P_1), g(P_2), \ldots, g(P_n))$ where $g = \Tr_{\F_{q^r}/ \F_{t}}(f)$. This implies the codeword $c$ is in the code $C(V, \Tr_{\F_{q^r}/ \F_{t}  }(L))$. The converse is similar. \end{proof}

\begin{defn}

The curve $X^u-\Tr_{\F_{q^r} / \F_q}  (Y)=0$ over $\F_{q^r}$ is known as the \emph{Extended Norm--Trace curve  over $\F_{q^r}$ associated to $q$,$r$ and $u$.}. Denote by 
$$\NT_u := \{ (x,y) \in \F_{q^r}^2 \ |  x^u = \Tr_{\F_{q^r}/ \F_{q}  }(y)  \}.$$

\end{defn}

The authors in \cite{BO07} note that $\# \NT_u = q^{r-1}(u(q-1)+1)$ and the genus is $\frac{(q^{r-1}-1)(u-1)}{2}.$ When $u = \frac{q^r-1}{q-1}$ the curve is a Norm--Trace curve. When $r=2$ the curve is a quasi--Hermitian curve.

\begin{defn}
\emph{The $(q^{r-1}, u)$--weight} of a monomial $X^iY^j$ is $$\rho_{q^{r-1}, u}(X^iY^j) := q^{r-1}i+uj.$$ We denote the set of all monomials whose $(q^{r-1}, u)$--weight is at most $s$ by $\mathcal{M}_{q^{r-1}, u}(s)$.
\end{defn}

\begin{defn}

For $s$, the Extended Norm--Trace code of weight $s$ is $$\NT_{u}(s) := C(\NT_u, \M_{q^{r-1}, u}(s)).$$

\end{defn}

\begin{prop}\cite{BO07}
$$\mathcal{NT}_u(s)^{\perp} = \mathcal{NT}_u(q^{r-1}(u-1) + u(q^{r-1}-1)-1-s).$$
\label{prop:DualENTCode}
\end{prop}


\section{Dimension of Subfield Subcodes}


Our aim is to find or bound the dimension of $\mathcal{NT}_u(s)|\F_t$. 

\begin{lem}
$$\dim\mathcal{NT}_u(s)|\F_t = n - \dim \Tr_{\F_{q^r}/ \F_t}(\mathcal{NT}_u(q^{r-1}(u-1) + u(q^{r-1}-1)-1-s))$$
\end{lem}
\begin{proof} The proof follows from Proposition \ref{prop:DualENTCode} and Delsarte's Theorem.\end{proof}

As $\mathcal{NT}_u(s)$ is an affine variety code, Lemma \ref{lem:traceaffinecode} implies $\Tr_{\F_{q^r}/ \F_t}(\mathcal{NT}_u(s))$ is also an affine variety code. We aim to determine or bound  $\dim(\Tr_{\F_{q^r}/ \F_t}(\mathcal{NT}_u(s))).$ To do this we need to study the kernel of the evaluation map of the functions in  $\sum\limits_{i=0}^{m-1}(\mathcal{M}_{q^{r-1}, u}(s)^{(t^i)}))$ on the points of $\NT_u$. For this purpose, we compute a Gr\"obner basis for the ideal corresponding to $\NT_u$. 

\begin{lem}
The ideal of polynomial functions which vanish on $\NT_u$ is generated by $X^u-\Tr_{\F_{q^r}/ \F_q}(Y)$ and $X^{u(q-1)+1}-X$.
\label{lem:NTIdeal}
\end{lem}

\begin{proof}Let $I$ denote the ideal of polynomial functions which vanish on $\NT_u$. This ideal is generated by $X^u-\Tr_{\F_{q^r}/ \F_q}(Y)$, $X^{q^r}-X$ and $Y^{q^r}-Y$. The polynomial $(X^u-\Tr_{\F_{q^r}/ \F_q}(Y))^q - X^{u} + \Tr_{\F_{q^r}/ \F_q}(Y) = X^{qu} - X^u -Y^{q^r}+Y \in I$. As both $X^{qu} - X^u \in I$ and $X^{q^r}-X \in I$, their greatest common factor, $X^{u(q-1)+1}-X$ is in the ideal. As $Y^{q^r}-Y$ is a combination of $X^u-\Tr_{\F_{q^r}/ \F_q}(Y)$ and $X^{u(q-1)+1}-X$ it follows that $I$ is generated by $X^u-\Tr_{\F_{q^r}/ \F_q}(Y)$ and $X^{u(q-1)+1}-X$.  \end{proof}

As the ideal of polynomial functions which vanish on $\NT_u$ is the ideal generated by $X^u-\Tr_{\F_{q^r}/ \F_q}(Y)$ and $X^{u(q-1)+1}-X$ we have the following corollary.

\begin{cor}

$$f(P_i) = g(P_i)\  \forall P_i \in \NT_u \makebox{ if and only if }f  - g \in \langle X^u-\mathcal{T}(Y), X^{u(q-1)+1}-X \rangle$$

\end{cor}
\begin{proof} Two polynomials evaluate to the same function on $\NT_u$ if and only if $f-g$ is in the ideal of polynomial functions which vanish on $\NT_u$. As this ideal equal to $\langle X^u-\mathcal{T}(Y), X^{u(q-1)+1}-X \rangle$ the proof now follows.\end{proof}






\begin{lem}
The $(q^{r-1}, u)$--weights of the support monomials of $X^u-\Tr_{\F_{q^r}/ \F_q}(Y)$ are congruent $\mod (q-1)u.$
Likewise the $(q^{r-1}, u)$--weights of the support of $X^{u(q-1)}-X$ are congruent $\mod (q-1)u.$
\label{lem:monweight}
\end{lem}
\begin{proof}Note that $\rho_{q^{r-1},u}(X^u) = q^ru$ and $\rho_{q^{r-1}u}(Y^{q^i}) = q^iu$. As $qu - u =  (q-1)u$ the result
 is true for $X^u-\Tr_{\F_{q^r}/ \F_q}(Y)$. The $(q^{r-1}, u)$--weight of $X^{u(q-1)+1}$ is $q^{r-1}(u(q-1)+1)$. \end{proof}

\begin{cor}

If the $(q^{r-1}, u)$--weights of the monomials of the nonzero terms of $f$ are congruent $\mod (q-1)u$, then the $(q^{r-1}, u)$--weights of the monomials of the nonzero terms of $f \mod  \langle  X^{u}-\mathcal{T}(Y), X^{q^r}-X \rangle$ are also congruent $\mod (q-1)u$.
\label{cor:remweight}
\end{cor}



In \cite{JP10} we used simpler techniques to prove that Lemmas \ref{lem:NTIdeal}, and \ref{lem:monweight} and Corollary \ref{cor:remweight} hold for Norm--Trace curves. As there are fewer monomials, but still a similar structure, the computations will be faster for extended Norm--Trace codes as the divisor $u$ decreases. As the authors state in \cite{BO07}, we can find a lower bound of the minimum distance of $\NT_u(s)$ from the theory of order domain codes.  We use the following minimum distance bound from \cite{G08}.




\begin{prop}[\cite{G08}]

Let $\Delta = \{ X^iY^j \ | \ 0 \leq i \leq u(q-1), 0 \leq j \leq q^{r-1}  \}$. Then the minimum distance of $\NT_u(s)$ is at least:

$$ \min\limits_{P \in \M_u(s)}\# \{ K \in \Delta \ | \ \exists K' \in \Delta \ s.t. \exists P \in \Delta \ s.t. \ \rho_u(K') + \rho_u(P) = \rho_u(K)  \}$$

\end{prop}

This lower bound is analogous to the BCH bound. This bound is an improvement on  the minimum distance of subfield subcodes. We use this to find the minimum distance of some subfield subcodes of extended Norm--Trace codes as in the following examples.

\section{Finding the codes}

In order to find the parameters of $\NT_u(s)|\F_t$ we use the minimum distance of the supercode $\NT_u(s)$ itself as a lower bound on the minimum distance of the subfield subcode. This bound is analogous to the BCH bound.

In order to determine the dimension it is quite difficult to work with $\NT_u(s)|\F_t$ or with $\NT_u(s)^o$. Delsarte's theorem allows us to work with the trace code of the dual code instead. As $\dim\Tr_{\F_{q^r}/\F_{t}}\NT_u(s') = \dim\NT_u(s')^\wedge$, we find the dimension of $\M_u(s')^\wedge$ reduced modulo the ideal of $\NT_u$.

\subsection{Binary Subcodes}

First we consider the extended Norm--Trace curve given by $X^3 +Y^8+Y^4+Y^2+Y$ over $\F_{16}$. We would like to find $\dim \NT_3(36)|\F_2$.  The dual code,$\NT_3(36)^\perp$, is equal to  $\NT_3(8)$. Note that $\M_3(8) = \{1,Y,Y^2, X\}.$ Thus $\M_3(8)^\wedge$ is spanned by $1, Y, Y^2, Y^4, Y^8, X, X^2,X^4, X^8$. Now we reduce these monomials modulo $\NT_3$, which is the ideal generated by the polynomials $X^4+X$ and $X^3 +Y^8+Y^4+Y^2+Y$. In this case we can see that $X^4$ is equivalent to $X$, $X^8$ is equivalent to $X^2$ and $Y^8$ is equivalent to $X^3+Y^4+Y^2+Y$. Therefore we find that $\Tr_{\F_{16}/\F_2}(\M_3(8))$ is generated by the independent polynomials: $\{1, Y, Y^2, Y^4, X^3, X, X^2\}.$ Thus $\dim \Tr_{\F_{16}/\F_2}(\NT_3(8)) = 7$ and for the subfield subcode $\dim\NT_3(36)|\F_2 = 25$. Following Geil \cite{G08}, the code $\NT_3(36)$ has minimum distance at least $3$. However as $\Tr_{\F_{16}/\F_2}(\NT_3(8))$ contains the all ones codeword, therefore all codewords of $\NT_3(36)|\F_2$ have even weights. Thus $\NT_3(36)|F_2$ is an optimal $[32, 25, 4]$ binary code. 

Now we consider the extended Norm--Trace curve $X^5 +Y^8+Y^4+Y^2+Y$ over $\F_{16}$. The code $\NT_5(65)$ has parameters $[48, 44, 3]$ over $\F_{16}$. The dual code is $\NT_5(65)^\perp = \NT_5(10)$. Now we shall find $\dim \Tr_{\F_{16}/\F_2}(\NT_5(10))$. In this case $\M_5(10) = \{1,Y,Y^2, X\}.$ As in the previous case $\M_5(10)^\wedge$ is spanned by the monomials $1, Y, Y^2, Y^4, Y^8, X, X^2,X^4, X^8$. By considering reductions modulo $X^6+X$ and $X^5 +Y^8+Y^4+Y^2+Y$ we find that $Y^8$ is equivalent to $X^3 + Y^4+Y^2+Y$ and $X^8$ is equivalent to $X^3$. Therefore the evaluation $Y^8$ is a linear combination of the evaluation of the other monomials and we obtain that $\Tr_{\F_{16}/\F_2}(\M_5(10))$ is generated by $\{1, Y, Y^2, Y^4, X, X^2, X^4,X^3\}.$ Thus $\dim Tr(\NT_5(10)) = 8$ and $\dim\NT_5(65)|\F_2  = 40$. As in the previous example, the code $\NT_5(65)$ has minimum distance at least $3$. However as $Tr(\NT_5(10))$ contains the all ones codeword, we obtain that $\NT_5(65)|F_2$ is an optimal $[48,40, 4]$ binary code. 

\subsection{Quaternary Subcodes}


For the curve $X^5 +Y^8+Y^4+Y^2+Y$ we will study the codes $\NT_5(60)$ and $\NT_5(62)$. Using Geil's bound \cite{G08} on the minimum distance we find these are $[48,43,3]$ and $[48,44,3]$ codes over $\F_{16}$ respectively. If one reduces the monomials in $\M_5(60)^{(4)}$ modulo the ideal $\NT_5$ one finds that the reductions of the $43$ monomials are contained in $\M_5(60)$. Therefore the code $\NT_5(60)$ is invariant under the Frobenius automorphism $x \mapsto x^4$. This implies $\NT_5(60) = \NT_5(60)^o$  and therefore there exists a $[48,43,3]$ optimal quaternary code. The reductions of the monomials in $\M_5(62)^{(4)}$ are also contained in $\M_5(62)$. The same argument implies there is also a $[48,44,3]$ quaternary code.

The binary and quaternary codes found in this section have optimal minimum distance for their given dimension or have the best known minimum distance \cite{Grassl:codetables}.

\end{document}